\documentclass[a4paper]{amsart}

\usepackage{amssymb}
\usepackage{latexsym}
\usepackage{amsmath}
\usepackage{euscript}
\usepackage{float}

\numberwithin{equation}{section}

\newcounter{counter_a}
\newenvironment{myenum}{\begin{list}{{\rm(\roman{counter_a})}}%
{\usecounter{counter_a}
\setlength{\itemsep}{1.ex}\setlength{\topsep}{1.5ex}
\setlength{\leftmargin}{5ex}\setlength{\labelwidth}{5ex}}}{\end{list}}
\newcommand\ds{\displaystyle}

\DeclareMathOperator{\dom}{dom}
\DeclareMathOperator{\ran}{ran}

\newcommand{\rmO}{{\rm O}}

\floatstyle{plain}
\newfloat{fig}{thp}{lob}
\floatname{fig}{Figure}
\numberwithin{fig}{section}

\newtheorem{thm}{Theorem}[section]

 \newtheorem{cor}[thm]{Corollary}
 \newtheorem{lem}[thm]{Lemma}
 \newtheorem{prop}[thm]{Proposition}
 \theoremstyle{definition}

 \newtheorem{dfn}[thm]{Definition}
 \newtheorem*{dfn*}{Definition}

 \theoremstyle{remark}
 \newtheorem{rem}[thm]{Remark}
 
\numberwithin{equation}{section}

\newcommand{\ba}{\begin{array}}
\newcommand{\ea}{\end{array}}
\newcommand{\bea}{\begin{eqnarray}}
\newcommand{\eea}{\end{eqnarray}}
\newcommand{\bead}{\begin{eqnarray*}}
\newcommand{\eead}{\end{eqnarray*}}
\newcommand{\be}{\begin{equation}}
\newcommand{\ee}{\end{equation}}
\newcommand{\bed}{\begin{displaymath}}
\newcommand{\eed}{\end{displaymath}}
\newcommand{\bl}{\begin{lem}}
\newcommand{\el}{\end{lem}}
\newcommand{\bp}{\begin{prop}}
\newcommand{\ep}{\end{prop}}
\newcommand{\bt}{\begin{thm}}
\newcommand{\et}{\end{thm}}

\newcommand{\bc}{\begin{cor}}
\newcommand{\ec}{\end{cor}}

\newcommand{\br}{\begin{rem}}
\newcommand{\er}{\end{rem}}
\newcommand{\bd}{\begin{defn}}
\newcommand{\ed}{\end{defn}}

\newcommand\cB{\mathcal B}

\newcommand\cG{\mathcal G}
\newcommand\cH{\mathcal H}
\newcommand\cK{\mathcal K}

\newcommand\cL{\mathcal L}

\newcommand\CC{\mathbb C}
\newcommand\NN{\mathbb N}

\newcommand\frA{\mathfrak A}

\newcommand\frS{\mathfrak S}

\newcommand\ov{\overline}

\newcommand{\defeq}{\mathrel{\mathop:}=}

\newcommand{\defequ}{\mathrel{\mathop:}\hspace*{-0.72ex}&=}

\newcommand\AD{A_{\rm D}}
\newcommand\AN{A_{\rm N}}

\newcommand\void[1]{}

\def\AD{A_{\rm D}}
\def\AN{A_{\rm N}}

\DeclareMathOperator\tr{tr}

\def\sS{{\mathfrak S}}

   \def\dN{{\mathbb N}}   
      \def\dR{{\mathbb R}}

   \def\cB{{\mathcal B}}   
      
\def\cG{{\mathcal G}}   \def\cH{{\mathcal H}}   
   \def\cK{{\mathcal K}}   \def\cL{{\mathcal L}}

\begin{document}

\title[Trace formulae and singular values]{Trace formulae and singular values of \\
resolvent power differences of \\ self-adjoint elliptic operators}

\author{Jussi Behrndt}
\address{Technische Universit\"{a}t Graz,
Institut f\"{u}r Numerische Mathematik \\
Steyrergasse 30,
8010 Graz, Austria}
\email{behrndt@tugraz.at}

\author{Matthias Langer}

\address{Department of Mathematics and Statistics,
University of Strathclyde,
26 Richmond Street, Glasgow G1 1XH, United Kingdom}
\email{m.langer@strath.ac.uk}

\author{Vladimir Lotoreichik}

\address{Technische Universit\"{a}t Graz,
Institut f\"{u}r Numerische Mathematik\\
Steyrergasse 30,
8010 Graz, Austria}
\email{lotoreichik@math.tugraz.at}

\begin{abstract}
In this note self-adjoint realizations of second order elliptic differential
expressions with non-local Robin boundary conditions on a domain $\Omega\subset\dR^n$
with smooth compact boundary are studied. A Schatten--von Neumann type estimate for
the singular values of the difference of the $m$th powers of the resolvents of two
Robin realizations is obtained, and for $m>\tfrac{n}{2}-1$ it is shown that the
resolvent power difference is a trace class operator. The estimates are slightly
stronger than the classical singular value estimates by M.\,Sh.~Birman where one
of the Robin realizations is replaced by the Dirichlet operator.
In both cases trace formulae are proved, in which the trace of the resolvent
power differences in $L^2(\Omega)$ is written in terms of the trace of derivatives of
Neumann-to-Dirichlet and Robin-to-Neumann maps on the boundary space $L^2(\partial\Omega)$.
\end{abstract}


\maketitle

\section{Introduction}

Let $\Omega\subset\dR^n$ be a bounded or unbounded domain with smooth compact boundary and
let $\cL$ be a formally symmetric second order elliptic differential expression with variable
coefficients defined on $\Omega$. As a simple example one may consider $\cL=-\Delta$
or $\cL=-\Delta+V$ with some real function $V$. Denote by $\AD$ the self-adjoint Dirichlet
operator associated with $\cL$ in $L^2(\Omega)$ and let $A_{[\beta]}$ be a self-adjoint
realization of $\cL$ in $L^2(\Omega)$ with Robin boundary conditions of the form
$\beta f|_{\partial\Omega}=\tfrac{\partial f}{\partial\nu}\vert_{\partial\Omega}$
for functions $f\in\dom A_{[\beta]}$. Here $\beta$ is a real-valued bounded function
on $\partial\Omega$; in the special case $\beta=0$ one obtains the Neumann operator $\AN$
associated with $\cL$.

Half a century ago it was observed by M.\,Sh.~Birman in his fundamental
paper \cite{B62} that the difference of the resolvents of $\AD$ and $A_{[\beta]}$
is a compact operator whose singular values $s_k$ satisfy
$s_k=\rmO\bigl(k^{-\frac{2}{n-1}}\bigr)$, $k\to\infty$, that is,
\begin{equation}\label{est1}
  (A_{[\beta]}-\lambda)^{-1}-(\AD-\lambda)^{-1}\in\sS_{\frac{n-1}{2},\infty},\qquad
  \lambda\in\rho(A_{[\beta]})\cap\rho(\AD),
\end{equation}
where $\sS_{p,\infty}$ denotes the weak Schatten--von Neumann ideal of order $p$;
for the latter see \eqref{def_Sp} below.
The difference of higher powers of the resolvents of $\AD$ and $A_{[\beta]}$
lead to stronger decay conditions of the form
\begin{equation}\label{est2}
  (A_{[\beta]}-\lambda)^{-m}-(\AD-\lambda)^{-m}\in\sS_{\frac{n-1}{2m},\infty},\qquad
  \lambda\in\rho(A_{[\beta]})\cap\rho(\AD);
\end{equation}
see, e.g.\ \cite{B62,G74,G84,G11-1,M10}.
The estimate \eqref{est1} for the decay of the singular values
is known to be sharp if $\beta$ is smooth, see \cite{BS80,G74,G84,G11-1}, and \cite{G11-2} for the case $\beta\in L^\infty(\partial\Omega)$;
the estimate \eqref{est2} is sharp for smooth $\beta$ by \cite{G84,G11-1}.
Observe that, for $m>\tfrac{n-1}{2}$,
the operator in \eqref{est2} belongs to the trace class ideal, and hence the wave operators
for the scattering pair $\{\AD,A_{[\beta]}\}$ exist and are complete, and the absolutely
continuous parts of $\AD$ and $A_{[\beta]}$ are unitarily equivalent.
A simple consequence of one of our main results in the present paper is the following representation
for the trace of the operator in \eqref{est2} (see Theorem~\ref{thm3}):
\begin{equation}\label{carron}
\begin{split}
  & \tr\bigl((A_{[\beta]}-\lambda)^{-m} - (\AD-\lambda)^{-m}\bigr)\\
  &\qquad = \frac{1}{(m-1)!}
  \tr\biggl(\frac{d^{m-1}}{d\lambda^{m-1}}
  \Bigl(\bigl(I- M(\lambda)\beta \bigr)^{-1}M(\lambda)^{-1}
  M'(\lambda)\Bigr)\biggr),
\end{split}
\end{equation}
where $M(\lambda)$ is the Neumann-to-Dirichlet map (i.e.\ the inverse of the Dirichlet-to-Neumann map)
associated with $\cL$;
see also \cite[Corollary~4.12]{BLL11} for $m=1$.
In the special case that $A_{[\beta]}$ is the Neumann operator $\AN$, that is $\beta=0$,
the above formula simplifies to
\begin{equation}\label{carron2}
  \tr\bigl((\AN-\lambda)^{-m} - (\AD-\lambda)^{-m}\bigr) = \frac{1}{(m-1)!}
  \tr\biggl(\frac{d^{m-1}}{d\lambda^{m-1}}\Big(M(\lambda)^{-1}M'(\lambda)\Big)\biggr),
\end{equation}
which is an analogue of \cite[Th\'eor\`eme~2.2]{Ca02} and reduces
to \cite[Corollary~3.7]{AB09} in the case $m=1$.  We point out that the
right-hand sides in \eqref{carron} and \eqref{carron2} consist of traces of
operators in the boundary space $L^2(\partial\Omega)$, whereas the left-hand sides
are traces of operators in $L^2(\Omega)$.
Some related reductions for ratios of Fredholm perturbation determinants
can be found in \cite{GMZ07}.
We also refer to \cite{GHSZ96} for other types of trace formulae for Schr\"odinger operators.

Recently, it was shown in \cite{BLLLP10} that if one considers two
self-adjoint Robin realizations $A_{[\beta_1]}$ and $A_{[\beta_2]}$ of $\cL$,
then the estimate \eqref{est1} can be improved to
\begin{equation}\label{est3}
  (A_{[\beta_1]}-\lambda)^{-1}-(A_{[\beta_2]}-\lambda)^{-1}\in\sS_{\frac{n-1}{3},\infty},
\end{equation}
so that, roughly speaking, any two Robin realizations with bounded coefficients $\beta_j$
are closer to each other than to the Dirichlet operator $\AD$; see also \cite{BLL11} and
the paper \cite{G11-2} by G.~Grubb where the estimate \eqref{est3} was shown to be sharp
under some smoothness conditions on the functions $\beta_1$ and $\beta_2$.
One of the main objectives of this note is to prove a counterpart of \eqref{est2} for
higher powers of resolvents of $A_{[\beta_1]}$ and $A_{[\beta_2]}$.  For that we apply
abstract boundary triple techniques from extension theory of symmetric operators
and a variant of Krein's formula which provides a convenient factorization of the resolvent
difference of two self-adjoint realizations of $\cL$; cf.\ \cite{BL07,BL12,BLL11} and
\cite{BGW09,DM91,GM08-1,GM11,G68,G12,M10,P08,P12} for related approaches.
Our tools allow us to consider general non-local Robin type realizations of $\cL$ of the form
\begin{equation}\label{abop}
\begin{split}
  A_{[B]}f &= \cL f,\\ \dom A_{[B]}&=\Bigl\{f\in H^{3/2}(\Omega): \cL f\in L^2(\Omega),\,\,
  B f\vert_{\partial\Omega}=\tfrac{\partial f}{\partial\nu}\bigl|_{\partial\Omega}\Bigr\},
\end{split}
\end{equation}
where $B$ is an arbitrary bounded self-adjoint operator in $L^2(\partial\Omega)$
and $H^{3/2}(\Omega)$ denotes the $L^2$-based Sobolev space of order $3/2$.
In the special case where $B$ is the multiplication operator with a bounded
real-valued function $\beta$ on $\partial\Omega$ the differential operator
in \eqref{abop} coincides with the usual corresponding Robin realization $A_{[\beta]}$
of $\cL$ in $L^2(\Omega)$. It is proved in Theorem~\ref{thm2} that for
two self-adjoint realizations $A_{[B_1]}$ and $A_{[B_2]}$ as in \eqref{abop}
the difference of the $m$th powers of the resolvents satisfies
\begin{equation*}\label{est4}
  (A_{[B_1]}-\lambda)^{-m}-(A_{[B_2]}-\lambda)^{-m}\in\sS_{\frac{n-1}{2m+1},\infty},\qquad \lambda\in\rho(A_{[B_1]})\cap\rho(A_{[B_2]}),
\end{equation*}
and if, in addition, $B_1-B_2$ belongs to some weak Schatten--von Neumann ideal, the
estimate improves accordingly.  Moreover, for $m>\tfrac{n}{2}-1$ the resolvent
difference is a trace class operator and for the trace we obtain
\begin{equation}\label{traceform}
\begin{aligned}
  & \tr\bigl((A_{[B_1]}-\lambda)^{-m}-(A_{[B_2]}-\lambda)^{-m}\bigr) \\[0.5ex]
  &= \frac{1}{(m-1)!}\tr\biggl[\frac{d^{m-1}}{d\lambda^{m-1}}
  \biggl(\!\bigl(I-B_1 M(\lambda)\bigr)^{-1}(B_1-B_2)\bigl(I- M(\lambda)B_2\bigr)^{-1}
  M'(\lambda)\!\biggr)\!\biggr].
\end{aligned}
\end{equation}
As in \eqref{carron} and \eqref{carron2} the right-hand side in \eqref{traceform}
consists of the trace of derivatives of Robin-to-Neumann and Neumann-to-Dirichlet maps
on the boundary $\partial\Omega$, so that \eqref{traceform} can be viewed as a reduction
of the trace in $L^2(\Omega)$ to the boundary space $L^2(\partial\Omega)$.

The paper is organized as follows. We first recall some necessary facts about singular values
and (weak) Schatten--von Neumann ideals in Section~\ref{ssec:Sp}. In Section~\ref{sec:qbt}
the abstract concept of quasi boundary triples, $\gamma$-fields and Weyl functions
from \cite{BL07} is briefly recalled.  Furthermore, we prove some preliminary results
on the derivatives of the $\gamma$-field and Weyl function, and we provide some Krein-type formulae
for the resolvent differences of self-adjoint extensions of a symmetric operator.
Section~\ref{sec:main} contains our main results on singular value estimates and traces of resolvent power differences
of Dirichlet, Neumann and non-local Robin realizations of $\cL$.
In Section~\ref{elliptic.prelim} the elliptic differential expression is defined and a
family of self-adjoint Robin realizations is parameterized with the help of
a quasi boundary triple. A detailed analysis of the smoothing properties of the derivatives
of the corresponding $\gamma$-field and Weyl function together with Krein-type
resolvent formulae and embeddings of Sobolev spaces then leads to the estimates and
trace formulae in Theorems~\ref{thm1}, \ref{thm2} and \ref{thm3}.

\section{Schatten--von Neumann ideals and quasi boundary triples}


This section starts with preliminary facts on singular values and (weak) Schatten--von Neumann ideals.
Furthermore, we review the concepts of quasi boundary triples, associated $\gamma$-fields and Weyl functions,
which are convenient abstract tools for the parameterization and spectral analysis of self-adjoint
realizations of elliptic differential expressions.

\subsection{Singular values and Schatten--von Neumann ideals}
\label{ssec:Sp}

Let $\cH$ and $\cK$ be Hilbert spaces.  We denote by $\cB(\cH,\cK)$ the space of
bounded operators from $\cH$ to $\cK$ and by $\sS_\infty(\cH,\cK)$ the space of
compact operators.  Moreover, we set $\cB(\cH)\defeq\cB(\cH,\cH)$ and
$\sS_\infty(\cH)\defeq\sS_\infty(\cH,\cH)$.

The \emph{singular values} (or \emph{$s$-numbers}) $s_k(K)$,
$k=1,2,\dots$, of a compact operator $K\in\sS_\infty(\cH,\cK)$
are defined as the eigenvalues of the
non-negative compact operator $(K^*K)^{1/2}\in\sS_\infty(\cH)$, which are
enumerated in non-increasing  order and with multiplicities taken into account. Note
that the singular values of $K$ and $K^*$ coincide:
$s_k(K)=s_k(K^*)$ for $k=1,2,\dots$; see, e.g.\ \cite[II.\S2.2]{GK69}.
Recall that, for $p>0$, the \emph{Schatten--von Neumann ideals} $\sS_p(\cH,\cK)$ and
\emph{weak Schatten--von Neumann ideals} $\sS_{p,\infty}(\cH,\cK)$ are defined by
\begin{equation}\label{def_Sp}
\begin{aligned}
  \sS_p(\cH,\cK) &\defeq \biggl\{K\in\sS_\infty(\cH,\cK)\colon\sum_{k=1}^\infty
  \bigl(s_k(K)\bigr)^p < \infty\biggr\}, \\
  \sS_{p,\infty}(\cH,\cK)& \defeq \Bigl\{K\in\sS_\infty(\cH,\cK)\colon s_k(K)
  = \rmO\bigl(k^{-1/p}\bigr),\,k\to\infty\Bigr\}.
\end{aligned}
\end{equation}
If no confusion can arise, the spaces $\cH$ and $\cK$ are suppressed and
we write $\sS_p$ and $\sS_{p,\infty}$. For $0<p' < p$  the inclusions
\begin{equation}\label{Sp_con_Spinf}
  \sS_{p} \subset \sS_{p,\infty} \quad\text{and}\quad \sS_{p',\infty}\subset\sS_p
\end{equation}
hold; for $s,t>0$ one has
\begin{equation}\label{prod_Sp}
  \sS_{\frac{1}{s}}\cdot\sS_{\frac{1}{t}}=
  \sS_{\frac{1}{s+t}}\quad\text{and}\quad\sS_{\frac{1}{s},\infty}\cdot\sS_{\frac{1}{t},\infty}=
  \sS_{\frac{1}{s+t},\infty},
\end{equation}
where a product of operator ideals is defined as the set of all products.
We refer the reader to \cite[III.\S7 and III.\S14]{GK69} and \cite[Chapter~2]{S05}
for a detailed study of the classes $\sS_p$ and $\sS_{p,\infty}$; see also \cite[Lemma~2.3]{BLL11}.
The ideal of \emph{nuclear} or \emph{trace class operators} $\sS_1$
plays an important role later on.
The trace of a compact operator $K\in\sS_1(\cH)$ is defined as
\begin{equation*}\label{deftrace}
  \tr K \defeq \sum_{k=1}^\infty \lambda_k(K),
\end{equation*}
where $\lambda_k(K)$ are the eigenvalues of $K$ and the sum converges absolutely.
It is well known (see, e.g.\ \cite[\S III.8]{GK69}) that, for $K_1,K_2\in\sS_1(\cH)$,
\begin{equation}\label{trace1}
  \tr(K_1 + K_2) = \tr K_1 + \tr K_2
\end{equation}
holds. Moreover, if $K_1\in\cB(\cH,\cK)$ and $K_2\in\cB(\cK,\cH)$ are such that
$K_1K_2\in\sS_1(\cK)$ and $K_2K_1\in\sS_1(\cH)$, then
\begin{equation}\label{trace2}
  \tr(K_1K_2) = \tr( K_2K_1).
\end{equation}

The next useful lemma can be found in, e.g.\ \cite{BLLLP10,BLL11} and is based
on the asymptotics of the eigenvalues of the Laplace--Beltrami operator.
For a smooth compact manifold $\Sigma$ we denote the usual $L^2$-based Sobolev
spaces by $H^r(\Sigma)$, $r\geq 0$.

\begin{lem}\label{le.s_emb}
Let $\Sigma$ be an $(n-1)$-dimensional  compact $C^\infty$-manifold without boundary, let
$\cK$ be a Hilbert space and $K\in\cB(\cK,H^{r_1}(\Sigma))$
with $\ran K \subset H^{r_2}(\Sigma)$ where $r_2>r_1\geq 0$.
Then $K$ is compact and its singular values $s_k(K)$ satisfy
\[
  s_k(K) = \rmO\bigl(k^{-\frac{r_2-r_1}{n-1}}\bigr), \quad k\to\infty,
\]
i.e.\ $K\in\sS_{\frac{n-1}{r_2-r_1},\infty}\bigl(\cK,H^{r_1}(\Sigma)\bigr)$ and
hence $K\in\sS_p\bigl(\cK,H^{r_1}(\Sigma)\bigr)$ for every $p>\frac{n-1}{r_2-r_1}$.
\end{lem}

\subsection{Quasi boundary triples and their Weyl functions}
\label{sec:qbt}

In this subsection we recall the definitions and some important properties of quasi
boundary triples, corresponding $\gamma$-fields and associated Weyl functions,
cf.\ \cite{BL07, BL12, BLL11} for more details.
Quasi boundary triples are particularly useful when dealing with elliptic boundary
value problems from an operator and extension theoretic point of view.

\begin{dfn}
\label{def:qbt}
Let $A$ be a closed, densely defined, symmetric operator in a Hilbert
space $(\cH, (\cdot,\cdot)_\cH)$.
A triple $\{\cG,\Gamma_0,\Gamma_1\}$ is called a \emph{quasi boundary triple}
for $A^*$ if $(\cG,(\cdot,\cdot)_\cG)$ is a Hilbert space and for some
linear operator $T\subset A^*$ with $\ov T  = A^*$ the following holds:
\begin{itemize}\setlength{\itemsep}{1.2ex}
\item [{\rm (i)}]
$\Gamma_0,\Gamma_1:\dom T\rightarrow\cG$ are linear mappings, and the
mapping $\Gamma \defeq \binom{\Gamma_0}{\Gamma_1}$ has dense range  in $\cG\times\cG$;
\item [{\rm (ii)}]
$A_0 \defeq T\upharpoonright\ker\Gamma_0$ is a self-adjoint operator in $\cH$;
\item [{\rm (iii)}]
for all $f,g\in \dom T$ the \emph{abstract Green identity} holds:
\begin{equation*}\label{green}
  (Tf,g)_{\cH}-(f,Tg)_{\cH}
  =(\Gamma_1 f,\Gamma_0 g)_{\cG}-(\Gamma_0 f,\Gamma_1 g)_{\cG}.
\end{equation*}
\end{itemize}
\end{dfn}

\medskip

\noindent
We remark that a quasi boundary triple for $A^*$ exists if and only if
the deficiency indices  of $A$ coincide.  Moreover, in the case of finite deficiency indices
a quasi boundary triple is automatically an ordinary boundary triple,
cf.\ \cite[Proposition~3.3]{BL07}. For the notion of (ordinary) boundary triples and their properties we refer to
\cite{Br76,DM91,DM95,GG91,Ko75}.
If  $\{\cG, \Gamma_0, \Gamma_1 \}$ is a quasi boundary triple for $A^*$,
then $A$ coincides with $T\upharpoonright\ker\Gamma$ and the operator
$A_1 \defeq T\upharpoonright\ker\Gamma_1$ is symmetric in $\cH$. We also
mention that a quasi boundary triple with the additional property $\ran\Gamma_0 =\cG$
is a generalized boundary triple in the sense of~\cite{DM95}; see
\cite[Corollary~3.7\,(ii)]{BL07}.

Next we recall the definition of the $\gamma$-field and the Weyl
function associated with the quasi boundary triple $\{\cG, \Gamma_0,
\Gamma_1 \}$ for $A^*$. Note that the decomposition
\[
  \dom T=\dom A_0\,\dot +\,\ker(T-\lambda)=\ker\Gamma_0\,\dot+\,\ker(T-\lambda)
\]
holds for all $\lambda\in\rho(A_0)$, so that $\Gamma_0\upharpoonright\ker(T-\lambda)$
is invertible for all $\lambda\in\rho(A_0)$. The (operator-valued) functions $\gamma$
and $M$ defined by
\begin{equation*}\label{gweyl}
  \gamma(\lambda) \defeq
  \bigl(\Gamma_0\upharpoonright\ker(T-\lambda)\bigr)^{-1}
  \quad\text{and}\quad
  M(\lambda) \defeq \Gamma_1\gamma(\lambda),\quad
  \lambda\in\rho(A_0),
\end{equation*}
are called the $\gamma$\emph{-field} and the \emph{Weyl function}
corresponding to the quasi boundary triple $\{\cG,\Gamma_0,\Gamma_1\}$.
These definitions coincide with the definitions of the $\gamma$-field and
the Weyl function in the case that $\{\cG,\Gamma_0,\Gamma_1\}$ is an
ordinary boundary triple, see~\cite{DM91}.
Note that, for each $\lambda \in \rho(A_0)$, the operator
$\gamma(\lambda)$ maps $\ran\Gamma_0\subset\cG$ into $\dom T\subset\cH$ and $M(\lambda)$ maps
$\ran\Gamma_0$ into $\ran\Gamma_1$. Furthermore, as an immediate consequence
of the definition of $M(\lambda)$, we obtain
\begin{equation*}\label{NDpropertyAbstr}
  M(\lambda) \Gamma_0 f_\lambda = \Gamma_1 f_\lambda, \qquad
  f_\lambda \in \ker (T-\lambda),\;\; \lambda \in \rho(A_0).
\end{equation*}

In the next proposition we collect some properties of the $\gamma$-field and the
Weyl function associated with the quasi boundary triple $\{\cG,\Gamma_0,\Gamma_1\}$ for $A^*$;
most statements were proved in \cite{BL07}.

\begin{prop} \label{gammaprop}
For all $\lambda,\mu\in \rho (A_0)$ the following assertions hold.
\begin{itemize}\setlength{\itemsep}{1.2ex}
\item[\rm (i)] 
The mapping $\gamma(\lambda)$ is a bounded, densely defined operator from $\cG$ into
$\cH$. The adjoint of $\gamma(\ov\lambda)$ has the representation
\[
  \gamma(\ov\lambda)^* = \Gamma_1(A_0-\lambda)^{-1} \in\cB(\cH,\cG).
\]
\item[\rm (ii)] 
The mapping $M(\lambda)$ is a densely defined (and in general unbounded) operator in $\cG$ that satisfies
$M(\lambda)\subset M(\ov\lambda)^*$ and
\begin{equation*}
  M(\lambda)h-M(\overline\mu)h=(\lambda-\overline\mu)\gamma(\mu)^*\gamma(\lambda)h
\end{equation*}
for all $h\in\cG_0$.
If\, $\ran\Gamma_0 = \cG$, then $M(\lambda)\in\cB(\cG)$ and $M(\lambda) = M(\ov\lambda)^*$.
\item[\rm (iii)] 
If $A_1=T\upharpoonright\ker\Gamma_1$ is a self-adjoint
operator in $\cH$ and $\lambda\in\rho(A_0)\cap\rho(A_1)$, then
$M(\lambda)$ maps $\ran \Gamma_0$ bijectively onto $\ran\Gamma_1$ and
\[
  M(\lambda)^{-1}\gamma(\ov\lambda)^* \in \cB(\cH,\cG).
\]
\end{itemize}
\end{prop}

\begin{proof}
Items (i), (ii) and the first part of (iii) follow from
\cite[Proposition~2.6\,(i), (ii), (iii), (v) and Corollary~3.7\,(ii)]{BL07}.
For the second part of (iii) note that $\{\cG,\Gamma_1,-\Gamma_0\}$ is also a quasi
boundary triple if $A_1$ is self-adjoint.  It is easy to see that in this case the
corresponding $\gamma$-field is $\widetilde\gamma(\lambda)=\gamma(\lambda)M(\lambda)^{-1}$.
Since $\ran(\gamma(\ov\lambda)^*)\subset\ran\Gamma_1$ by item (ii), the
operator $M(\lambda)^{-1}\gamma(\ov\lambda)^*$ is defined on $\cH$.  Now the boundedness
of $\widetilde\gamma(\lambda)$, which follows from (i), and the relation
$M(\lambda)\subset M(\ov\lambda)^*$ imply that $M(\lambda)^{-1}\gamma(\ov\lambda)^*$
is bounded.
\end{proof}

In the following we shall often use product rules for holomorphic operator-valued
functions.  Let $\cH_i$, $i=1,\dots,4$, be Hilbert spaces, $U$ a domain in $\CC$
and let $A\colon U\to\cB(\cH_3,\cH_4)$, $B\colon U\to\cB(\cH_2,\cH_3)$,
$C\colon U\to\cB(\cH_1,\cH_2)$ be holomorphic operator-valued functions.  Then
\begin{align}
  \label{rule1}
  \frac{d^m}{d\lambda^m}\bigl(A(\lambda)B(\lambda)\bigr)
  &= \sum_{\substack{p+q = m \\[0.2ex] p,q \ge 0}}
  \binom{m}{p}A^{(p)}(\lambda)B^{(q)}(\lambda), \\[1ex]
  \label{rule2}
  \frac{d^m}{d\lambda^m}\bigl(A(\lambda)B(\lambda)C(\lambda)\bigr)
  &= \sum_{\substack{p+q+r = m \\[0.2ex] p,q,r \ge 0}}
  \frac{m!}{p!\,q!\,r!}A^{(p)}(\lambda)B^{(q)}(\lambda)C^{(r)}(\lambda)
\end{align}
for $\lambda\in U$.
If $A(\lambda)^{-1}$ is invertible for every $\lambda\in U$, then
relation \eqref{rule1} implies the following formula for the derivative
of the inverse,
\begin{equation}\label{rule3}
  \frac{d}{d\lambda}\bigl(A(\lambda)^{-1}\bigr) = -A(\lambda)^{-1}A'(\lambda)A(\lambda)^{-1}.
\end{equation}
In the next lemma we consider higher derivatives of the $\gamma$-field and the Weyl function
associated with a quasi boundary triple $\{\cG,\Gamma_0,\Gamma_1\}$.

\begin{lem}\label{le:der_g_M}
For all $\lambda\in\rho(A_0)$ and all $k\in\dN$ the following holds.
\begin{myenum}
\item
$\ds\frac{d^k}{d\lambda^k}\gamma(\ov\lambda)^* = k!\,\gamma(\ov\lambda)^*(A_0-\lambda)^{-k}$;
\item
$\ds\frac{d^k}{d\lambda^k}\ov{\gamma(\lambda)} = k!(A_0-\lambda)^{-k}\ov{\gamma(\lambda)}$;
\item
$\ds\ov{\frac{d^k}{d\lambda^k}M(\lambda)} = \frac{d^{k-1}}{d\lambda^{k-1}}\bigl(\gamma(\ov\lambda)^*\ov{\gamma(\lambda)}\,\bigr) = k!\,\gamma(\ov\lambda)^*(A_0-\lambda)^{-(k-1)}\ov{\gamma(\lambda)}$.
\end{myenum}
\end{lem}

\begin{proof}
(i) We prove the statement by induction.  For $k=1$ we have
\begin{align*}
  \frac{d}{d\lambda}\gamma(\ov\lambda)^*
  &= \lim_{\mu\to\lambda}\frac{1}{\mu-\lambda}\bigl(\gamma(\ov\mu)^*-\gamma(\ov\lambda)^*\bigr) \\[1ex]
  &= \lim_{\mu\to\lambda}\frac{1}{\mu-\lambda}\Gamma_1\bigl((A_0-\mu)^{-1}-(A_0-\lambda)^{-1}\bigr) \\[1ex]
  &= \lim_{\mu\to\lambda}\Gamma_1(A_0-\mu)^{-1}(A_0-\lambda)^{-1}
  = \lim_{\mu\to\lambda}\gamma(\ov\mu)^*(A_0-\lambda)^{-1} \\[1ex]
  &= \gamma(\ov\lambda)^*(A_0-\lambda)^{-1},
\end{align*}
where we used Proposition~\ref{gammaprop}\,(i).
If we assume that the statement is true for $k\in\dN$, then
\begin{align*}
  \frac{d^{k+1}}{d\lambda^{k+1}}\gamma(\ov\lambda)^*
  &= k!\frac{d}{d\lambda}\Bigl(\gamma(\ov\lambda)^*(A_0-\lambda)^{-k}\Bigr) \\[1ex]
  &= k!\biggl[\Bigl(\frac{d}{d\lambda}\gamma(\ov\lambda)^*\Bigr)(A_0-\lambda)^{-k}
    + \gamma(\ov\lambda)^*\frac{d}{d\lambda}(A_0-\lambda)^{-k}\biggr] \\[1ex]
  &= k!\biggl[\gamma(\ov\lambda)^*(A_0-\lambda)^{-1}(A_0-\lambda)^{-k}
    + \gamma(\ov\lambda)^* k(A_0-\lambda)^{-k-1}\biggr] \\[1ex]
  &= k!(1+k)\gamma(\ov\lambda)^*(A_0-\lambda)^{-(k+1)},
\end{align*}
which proves the statement in (i) by induction.

(ii) This assertion is obtained from (i) by taking adjoints.

(iii) It follows from Proposition~\ref{gammaprop}\,(ii) that,
for $f\in\dom M(\lambda) = \ran \Gamma_0$,
\[
  \frac{d}{d\lambda}M(\lambda)f
  = \lim_{\mu\to\lambda}\frac{1}{\mu-\lambda}\bigl(M(\mu)-M(\lambda)\bigr)f
  = \lim_{\mu\to\lambda}\gamma(\ov\lambda)^*\gamma(\mu)f
  = \gamma(\ov\lambda)^*\gamma(\lambda)f.
\]
By taking closures we obtain the claim for $k=1$.
For $k\ge2$ we use \eqref{rule1} to get
\begin{align*}
  & \ov{\frac{d^k}{d\lambda^k}M(\lambda)}
  = \frac{d^{k-1}}{d\lambda^{k-1}}\Bigl(\gamma(\ov\lambda)^*\ov{\gamma(\lambda)}\,\Bigr)
  = \sum_{\substack{p+q = k-1 \\[0.2ex] p,q \ge 0}} \binom{k-1}{p}
    \biggl(\frac{d^{p}}{d\lambda^{p}}\gamma(\ov\lambda)^*\biggr)
    \frac{d^q}{d\lambda^q}\ov{\gamma(\lambda)} \\[1ex]
  &= \sum_{\substack{p+q = k-1 \\[0.2ex] p,q \ge 0}}
    \binom{k-1}{p}p!\,\gamma(\ov\lambda)^*(A_0-\lambda)^{-p}
    q!\,(A_0-\lambda)^{-q}\,\ov{\gamma(\lambda)} \\[1ex]
  &= \sum_{\substack{p+q = k-1 \\[0.2ex] p,q \ge 0}}
    (k-1)!\gamma(\ov\lambda)^*(A_0-\lambda)^{-(k-1)}\ov{\gamma(\lambda)}
  = k!\gamma(\ov\lambda)^*(A_0-\lambda)^{-(k-1)}\ov{\gamma(\lambda)},
\end{align*}
which finishes the proof.
\end{proof}

The following theorem provides a Krein-type formula for the resolvent difference
of $A_0$ and $A_1$ if $A_1$ is self-adjoint.  The theorem follows from
\cite[Corollary~3.11\,(i)]{BL07} with $\Theta=0$.

\begin{thm}
\label{thm:01}
Let $A$ be a closed, densely defined, symmetric operator in a Hilbert space $\cH$ and let
$\{\cG,\Gamma_0,\Gamma_1\}$ be a quasi boundary triple for $A^*$
with $A_0=T\upharpoonright\ker\Gamma_0$, $\gamma$-field $\gamma$ and Weyl function $M$.
Assume that $A_1=T\upharpoonright\ker\Gamma_1$ is self-adjoint in $\cH$. Then
\[
  (A_0-\lambda)^{-1} - (A_1-\lambda)^{-1} = \gamma(\lambda)M(\lambda)^{-1}\gamma(\ov\lambda)^*
\]
holds for $\lambda\in\rho(A_1)\cap\rho(A_0)$.
\end{thm}

\noindent
Note that the operator $M(\lambda)^{-1}\gamma(\ov\lambda)^*$ in Theorem~\ref{thm:01} above is bounded by Proposition~\ref{gammaprop}\,(iii).

In the following we deal with extensions of $A$, which are restrictions of $T$
corresponding to some abstract boundary condition. For a linear operator $B$ in $\cG$
we define
\begin{equation}\label{AB0}
  A_{[B]}f \defeq Tf, \quad \dom A_{[B]} \defeq \bigl\{ f\in \dom T\colon B\Gamma_1 f
  = \Gamma_0 f\bigr\}.
\end{equation}
In contrast to ordinary boundary triples, self-adjointness of the parameter $B$
does not imply self-adjointness of the corresponding extension $A_{[B]}$ in general.
The next theorem provides a useful sufficient condition for this
and a variant of Krein's formula, which will be used later;
see \cite[Corollary 6.18 and Theorem 6.19]{BL12} or \cite[Corollary~3.11, Theorem~3.13 and Remark~3.14]{BLL11}.

\begin{thm}
\label{thm:sa}
Let $A$ be a closed, densely defined, symmetric operator in a Hilbert space $\cH$ and let
$\{\cG,\Gamma_0,\Gamma_1\}$ be a quasi boundary triple for $A^*$
with $A_0=T\upharpoonright\ker\Gamma_0$, $\gamma$-field $\gamma$ and Weyl function $M$.
Assume that $\ran \Gamma_0 = \cG$, $A_1=T\upharpoonright\ker\Gamma_1$ is self-adjoint in $\cH$
and that $M(\lambda_0)\in\sS_\infty(\cG)$ for some $\lambda_0\in\rho(A_0)$.

If $B$ is a bounded self-adjoint operator in $\cG$, then the corresponding extension
$A_{[B]}$ is self-adjoint in $\cH$ and
\begin{align*}
  (A_{[B]} - \lambda)^{-1} - (A_0 - \lambda)^{-1}
  &= \gamma(\lambda)\big(I - B M(\lambda)\big)^{-1} B\gamma(\ov\lambda)^*\\
  & =  \gamma(\lambda)B\big(I - M(\lambda)B\big)^{-1} \gamma(\ov\lambda)^*
\end{align*}
holds for $\lambda\in\rho(A_{[B]})\cap\rho(A_0)$ with
\[
  \bigl(I-B M(\lambda)\bigr)^{-1},\,\bigl(I- M(\lambda)B\bigr)^{-1}\in\cB(\cG).
\]
\end{thm}

\section{Elliptic operators on domains with compact boundaries}\label{sec:main}

In this section we study self-adjoint realizations of elliptic second-order differential expressions
on a bounded or an exterior domain subject to Robin or more general non-local boundary conditions.
With the help of quasi boundary triple techniques we express the resolvent power differences of
different self-adjoint realizations in Krein-type formulae.  Using a detailed analysis
of the perturbation term together with smoothing properties of the derivatives of
the $\gamma$-fields and Weyl function we then obtain singular value estimates and trace formulae.

\subsection{Self-adjoint elliptic operators with non-local Robin boundary conditions}
\label{elliptic.prelim}

Let $\Omega\subset\dR^n$, $n\ge 2$, be a bounded or unbounded domain with a compact
$C^\infty$-boundary $\partial \Omega$. We denote by $(\cdot,\cdot)$ and $(\cdot,\cdot)_{\partial\Omega}$
the inner products in the Hilbert spaces $L^2(\Omega)$ and $L^2(\partial\Omega)$, respectively.
Throughout this section we consider a formally symmetric second-order elliptic differential expression
\begin{equation*}
  (\cL f)(x) \defeq -\sum_{j,k=1}^n
  \partial_j\big(a_{jk}\partial_k f\big)(x)+ a(x)f(x),
  \quad x\in\Omega,
\end{equation*}
with bounded infinitely differentiable, real-valued coefficients
$a_{jk}, a\in C^\infty(\overline\Omega)$ satisfying $a_{jk}(x)=a_{kj}(x)$ for
all $x\in\overline\Omega$ and $j,k=1,\dots,n$.
We assume that the first partial derivatives of the coefficients $a_{jk}$ are bounded in $\Omega$.
Furthermore, $\cL$ is assumed to be uniformly elliptic, i.e.\ the condition
\begin{equation*}
  \sum_{j,k=1}^n a_{jk}(x)\xi_j\xi_k\geq C\sum_{k=1}^n\xi_k^2
\end{equation*}
holds for some $C>0$, all $\xi=(\xi_1,\dots,\xi_n)^\top\in\dR^n$ and $x\in\overline
\Omega$.

For  a function $f\in C^\infty(\ov\Omega)$ we denote the trace by $f|_{\partial\Omega}$ and
the (oblique) Neumann trace by
\[
  \partial_\cL f|_{\partial\Omega} \defeq
  \sum_{j,k=1}^n a_{jk} \nu_j \partial_k f |_{\partial\Omega},
\]
with the normal vector field $\vec \nu = (\nu_1,\nu_2,\dots, \nu_n)$
pointing outwards $\Omega$.
By continuity, the trace and the Neumann trace can be extended to mappings
from $H^s(\Omega)$ to $H^{s-\frac{1}{2}}(\partial\Omega)$ for $s>\tfrac{1}{2}$ and $H^{s-\frac{3}{2}}(\partial\Omega)$ for $s>\tfrac{3}{2}$ ,
respectively.

Next we define a quasi boundary triple for the adjoint $A^*$ of the minimal operator
\[
  Af = \cL f, \quad \dom A = \bigl\{f\in H^2(\Omega): f |_{\partial\Omega}
  = \partial_\cL f|_{\partial\Omega}=0\bigr\}
\]
associated with $\cL$ in $L^2(\Omega)$. Recall that $A$ is a closed, densely defined,
symmetric operator with equal infinite deficiency indices and that
\[
  A^*f=\cL f,\quad \dom A^*=\{f\in L^2(\Omega):\cL f\in L^2(\Omega)\}
\]
is the maximal operator associated with $\cL$; see, e.g.\ \cite{ADN59, B65}.
As the operator $T$ appearing in the definition of a quasi boundary triple we choose
\begin{equation*}\label{HsL}
  Tf=\cL f,\quad
  \dom T = H_{\cL}^{3/2}(\Omega)
  \defeq \bigl\{f\in H^{3/2}(\Omega)\colon \cL f\in L^2(\Omega)\bigr\}
\end{equation*}
and we consider the boundary mappings
\begin{equation*}\label{eq:bmap}
\begin{alignedat}{2}
  & \Gamma_0 \colon \dom T \rightarrow L^2(\partial \Omega),\qquad
  & \Gamma_0 f \defequ \partial_\cL f|_{\partial\Omega},\\
  & \Gamma_1 \colon \dom T \rightarrow L^2(\partial \Omega),\qquad
  & \Gamma_1 f \defequ f|_{\partial\Omega}.
\end{alignedat}
\end{equation*}
Note that the trace and the Neumann trace can be extended to mappings from $H^{3/2}_{\cL}(\Omega)$
into $L^2(\partial\Omega)$.
With this choice of $T$ and $\Gamma_0$ and $\Gamma_1$ we have the following proposition.

\begin{prop}
\label{prop:qbt}
The triple $\{L^2(\partial\Omega), \Gamma_0,\Gamma_1\}$ is a quasi boundary triple for $A^*$
with the Neumann and Dirichlet operator as self-adjoint operators corresponding to the kernels of the boundary mappings,
\begin{equation}
\label{ADAN}
\begin{alignedat}{2}
  A_{\rm N} \defequ T\upharpoonright \ker\Gamma_0,\qquad &\dom A_{\rm N} &= \big\{ f\in H^2(\Omega)\colon \partial_\cL f|_{\partial\Omega} = 0 \big\},\\
  A_{\rm D} \defequ T\upharpoonright \ker\Gamma_1,\qquad &\dom A_{\rm D} &= \big\{ f\in H^2(\Omega)\colon f|_{\partial\Omega} = 0 \big\}.
\end{alignedat}
\end{equation}
The ranges of the boundary mappings are
\[
  \ran \Gamma_0 = L^2(\partial\Omega)\quad\text{and}\quad\ran\Gamma_1 = H^1(\partial\Omega),
\]
and the $\gamma$-field and Weyl function associated with $\{L^2(\partial\Omega), \Gamma_0,\Gamma_1\}$ are given by
\[
  \gamma(\lambda)\varphi =f_\lambda\quad\text{and}\quad
  M(\lambda)\varphi = f_\lambda|_{\partial\Omega},\qquad\lambda\in\rho(A_{\rm N}),
\]
for $\varphi\in L^2(\partial\Omega)$ where $f_\lambda\in H^{3/2}_\cL(\Omega)$ is the unique solution of the
boundary value problem $\cL u=\lambda u$, $\partial_\cL u|_{\partial\Omega}=\varphi$.
\end{prop}

We remark that the quasi boundary triple $\{L^2(\partial\Omega), \Gamma_0,\Gamma_1\}$ in
Proposition~\ref{prop:qbt} is a generalized boundary triple in the sense of \cite{DM95}
since the boundary mapping $\Gamma_0$ is surjective.

\begin{proof}
The proof of Proposition~\ref{prop:qbt} proceeds in the same way as the proof of \cite[Theorem~4.2]{BLL11},
except that here $T$ is defined on the larger space $H_{\cL}^{3/2}(\Omega)$.
Therefore we do not repeat the arguments here, but provide only the main references that
are necessary to translate the proof of \cite[Theorem~4.2]{BLL11} to the present situation.
The self-adjointness of $\AD$ and $\AN$ is ensured by \cite[Theorem~7.1\,(a)]{B65}
and \cite[Theorem~5\,(iii)]{B60}.
The trace theorem from \cite[Chapter~2, \S 7.3]{LM72} and the corresponding
Green identity (see, e.g.\ \cite[proof of Theorem~4.2]{BLL11})
yield the asserted properties of the ranges of the boundary mappings
$\Gamma_0$ and $\Gamma_1$ and the abstract Green identity in Definition~\ref{def:qbt}.
Hence \cite[Theorem 2.3]{BL07} implies that the triple $\{L^2(\partial\Omega), \Gamma_0,\Gamma_1\}$ in
Proposition~\ref{prop:qbt} is a quasi boundary triple for $A^*$;
cf.\ \cite[Theorem~3.2, Theorem~4.2 and Proposition~4.3]{BLL11} for further details.
\end{proof}

The space $H^s_{\rm loc}(\ov{\Omega})$, $s\ge 0$, consists of all measurable functions $f$
such that for any bounded open subset $\Omega'\subset\Omega$ the
condition $f\upharpoonright \Omega' \in H^s(\Omega')$ holds.
Since $\Omega$ is a bounded domain or an exterior domain and $\partial\Omega$ is compact,
any function in $H^s_{\rm loc}(\ov{\Omega})$ is $H^s$-smooth up to the boundary $\partial\Omega$.
For $f\in H^s_{\rm loc}(\ov{\Omega})\cap L^2(\Omega)$, $s\ge 0$, our assumptions on the coefficients in the differential expression $\cL$
imply that
\begin{equation}\label{elliptic}
\begin{split}
  &(\AD-\lambda)^{-1}f \in  H^{s+2}_{\rm loc}(\ov{\Omega})\cap  L^2(\Omega), \quad \lambda\in\rho(\AD), \\[1ex]
  &(\AN-\lambda)^{-1}f \in  H^{s+2}_{\rm loc}(\ov{\Omega})\cap L^2(\Omega), \quad \lambda\in\rho(\AN).
\end{split}
\end{equation}
These smoothing properties can be easily deduced from \cite[Theorem~4.18]{McL00},
where they are formulated and proved in the language of boundary value problems.

The operators $\gamma(\lambda)$ and $M(\lambda)$ are also called
\emph{Poisson operator} and \emph{Neumann-to-Dirichlet map}
for the differential expression $\cL-\lambda$.  From Proposition~\ref{gammaprop}
various properties of these operators can be deduced. In the next lemma we
collect smoothing properties of these operators, which follow, basically, from
Proposition~\ref{gammaprop} and the trace theorem for Sobolev spaces on smooth domains
and its generalizations given in \cite[Chapter~2]{LM72}.

\begin{lem}
\label{lem:gammaWeyl}
Let $\{L^2(\partial\Omega), \Gamma_0,\Gamma_1\}$ be the quasi boundary triple from Proposition~\ref{prop:qbt}
with $\gamma$-field $\gamma$ and Weyl function $M$. Then, for all $s\ge 0$, the following statements hold.
\begin{itemize}\setlength{\itemsep}{1.2ex}
\item[{\rm (i)}] $\ran\bigl(\gamma(\lambda)\upharpoonright H^s(\partial\Omega)\bigr)
  \subset H^{s+\frac32}_{\rm loc}(\ov{\Omega})\cap L^2(\Omega)$ for all $\lambda \in \rho(\AN)$;
\item[{\rm (ii)}] $\ran\bigl(\gamma(\ov\lambda)^*\upharpoonright  H^s_{\rm loc}(\ov{\Omega})\cap L^2(\Omega)\bigr)
  \subset H^{s+\frac32}(\partial\Omega)$ for all  $\lambda \in \rho(\AN)$;
\item[{\rm (iii)}] $\ran\bigl(M(\lambda)\upharpoonright H^s(\partial\Omega)\bigr)
  \subset H^{s+1}(\partial\Omega)$ for all $\lambda \in \rho(\AN)$;
\item[{\rm (iv)}] $\ran\bigl(M(\lambda)\upharpoonright H^s(\partial\Omega)\bigr)
  = H^{s+1}(\partial\Omega)$ for all $\lambda \in \rho(\AD) \cap \rho(\AN)$.
\end{itemize}
\end{lem}

\begin{proof}
(i) It follows from the decomposition $\dom T=\dom A_{\rm N}\dotplus\ker(T-\lambda)$, $\lambda\in\rho(\AN)$,
and the properties of the Neumann trace \cite[Chapter 2, \S 7.3]{LM72}
that the restriction of the mapping $\Gamma_0$ to
\begin{equation*}
  \ker(T-\lambda)\cap H^{s+\frac32}_{\rm loc}(\ov{\Omega})
\end{equation*}
is a bijection onto $H^s(\partial\Omega)$, $s\geq 0$.
Hence, by the definition of the $\gamma$-field, we obtain
\begin{equation*}
  \ran \bigl(\gamma(\lambda)\upharpoonright H^s(\partial\Omega)\bigr)
  =  \ker(T-\lambda)\cap H^{s+\frac32}_{\rm loc}(\ov{\Omega}) \subset  H^{s+\frac32}_{\rm loc}(\ov{\Omega})\cap L^2(\Omega).
\end{equation*}

(ii) According to Proposition~\ref{gammaprop}\,(i) and the definition of $\Gamma_1$ we have
\begin{equation*}
  \gamma(\ov\lambda)^* = \Gamma_1 (A_{\rm N}-\lambda)^{-1}.
\end{equation*}
Employing \eqref{elliptic} and the properties of the Dirichlet trace \cite[Chapter 2, \S 7.3]{LM72}
we conclude that
\begin{equation*}
  \ran \bigl(\gamma(\ov\lambda)^*\upharpoonright H^s_{\rm loc}(\ov{\Omega})\cap L^2(\Omega)\bigr) \subset H^{s+\frac32}(\partial\Omega)
\end{equation*}
holds for all $s\geq 0$.

Assertion (iii) follows from the definition of $M(\lambda)$, item (i), the fact that
$\Gamma_1$ is the Dirichlet trace operator and properties of the latter.

To verify (iv) let $\psi\in H^{s+1}(\partial\Omega)$. Since $\lambda\in\rho(A_{\rm D})$,
we have the decomposition $\dom T=\dom A_{\rm D}\dotplus\ker(T-\lambda)$ and
there exists a unique function $f_\lambda\in\ker(T-\lambda)\cap H^{s+\frac32}_{\rm loc}(\ov{\Omega})$
such that $f_\lambda|_{\partial\Omega}=\psi$.  Hence
\begin{equation*}
  \Gamma_0 f_\lambda = \varphi\in H^s(\partial\Omega)
  \qquad\text{and}\qquad M(\lambda)\varphi=\psi,
\end{equation*}
that is, $H^{s+1}(\partial\Omega)\subset\ran\bigl(M(\lambda)\upharpoonright H^s(\partial\Omega)\bigr)$,
and (iii) implies the assertion.
\end{proof}

In the next proposition we list some weak Schatten--von Neumann ideal properties
of the derivatives of the $\gamma$-field and Weyl function, which follow from Lemma~\ref{le:der_g_M},
elliptic regularity and Lemma~\ref{le.s_emb}.

\begin{prop}
\label{prop:Sp_g_M}
Let $\{L^2(\partial\Omega), \Gamma_0,\Gamma_1\}$ be the quasi boundary triple
from Proposition~\ref{prop:qbt} with $\gamma$-field $\gamma$ and Weyl function $M$.
Then the following statements hold.
\begin{itemize}\setlength{\itemsep}{1.2ex}
\item[\rm (i)] For all $\lambda\in\rho(A_{\rm N})$ and $k\in\dN_0$,
\begin{equation}\label{SvN_dergamma}
\begin{split}
  &\frac{d^k}{d\lambda^k}\gamma(\lambda)\in
  \sS_{\frac{n-1}{2k+3/2},\infty}\big(L^2(\partial\Omega), L^2(\Omega)\big), \\[1ex]
  &\frac{d^k}{d\lambda^k}\gamma(\ov\lambda)^*\in
  \sS_{\frac{n-1}{2k+3/2},\infty}\big(L^2(\Omega), L^2(\partial\Omega)\big).
\end{split}
\end{equation}
\item[\rm (ii)] For all $\lambda\in\rho(A_{\rm N})$ and $k\in\dN_0$,
\[
  \frac{d^k}{d\lambda^k}M(\lambda)\in\sS_{\frac{n-1}{2k+1},\infty}\big(L^2(\partial\Omega)\big).
\]
\end{itemize}
\end{prop}

\begin{proof}
(i)
Let $\lambda\in\rho(\AN)$ and $k\in\NN_0$.
It follows from \eqref{elliptic} that $\ran\bigl((\AN-\lambda)^{-k}\bigr)\subset
H^{2k}_{\rm loc}(\ov{\Omega})\cap L^2(\Omega)$
and hence from Lemma~\ref{lem:gammaWeyl}\,(ii) that
\[
  \ran\bigl(\gamma(\ov\lambda)^*(\AN-\lambda)^{-k}\bigr) \subset H^{2k+3/2}(\partial\Omega).
\]
Thus Lemma~\ref{le.s_emb} with $\cK=L^2(\Omega)$, $\Sigma=\partial\Omega$, $r_1=0$ and $r_2=2k+3/2$
implies that
\begin{equation}\label{gres_SvN}
  \gamma(\ov\lambda)^*(\AN-\lambda)^{-k}
  \in \frS_{\frac{n-1}{2k+3/2},\infty}\bigl(L^2(\Omega),L^2(\partial\Omega)\bigr).
\end{equation}
By taking the adjoint in \eqref{gres_SvN} and replacing $\lambda$ by $\ov\lambda$ we obtain
\begin{equation}\label{resg_SvN}
  (\AN-\lambda)^{-k}\gamma(\lambda)
  \in \frS_{\frac{n-1}{2k+3/2},\infty}\bigl(L^2(\partial\Omega),L^2(\Omega)\bigr).
\end{equation}
Now from Lemma~\ref{le:der_g_M}\,(i) and (ii) and \eqref{gres_SvN} and \eqref{resg_SvN}
we obtain \eqref{SvN_dergamma}.

(ii) For $k=0$ we observe that, by Lemma~\ref{lem:gammaWeyl}\,(iii), we
have $\ran M(\lambda)\subset H^1(\partial\Omega)$.
Therefore Lemma~\ref{le.s_emb} with $\cK=L^2(\partial\Omega)$, $\Sigma=\partial\Omega$, $r_1=0$
and $r_2=1$ implies that $M(\lambda)\in\frS_{n-1,\infty}(L^2(\partial\Omega))$.
For $k\ge1$ we have
\[
  \frac{d^k}{d\lambda^k}M(\lambda) = k!\,\gamma(\ov\lambda)^*(A_{\rm N}-\lambda)^{-(k-1)}\gamma(\lambda)
\]
from Lemma~\ref{le:der_g_M}\,(iii).  Hence \eqref{gres_SvN} and \eqref{resg_SvN} imply that
\[
  \frac{d^k}{d\lambda^k}M(\lambda)
  \in \sS_{\frac{n-1}{2(k-1)+3/2},\infty}\cdot\sS_{\frac{n-1}{3/2},\infty}
  = \sS_{\frac{n-1}{2k+1},\infty},
\]
where the last equality follows from \eqref{prod_Sp}.
\end{proof}

As a consequence of Theorem~\ref{thm:01} we obtain a factorization for the
resolvent difference of self-adjoint operators $A_{\rm N}$ and $A_{\rm D}$.

\begin{cor}
\label{cor:01}
Let $\{L^2(\partial\Omega), \Gamma_0,\Gamma_1\}$ be the quasi boundary triple from
Proposition~\ref{prop:qbt} with $\gamma$-field $\gamma$ and Weyl function $M$.  Then
\begin{equation*}
  (\AN-\lambda)^{-1} - (\AD-\lambda)^{-1} = \gamma(\lambda)M(\lambda)^{-1}\gamma(\ov\lambda)^*
\end{equation*}
holds for $\lambda\in\rho(\AD)\cap\rho(\AN)$.
\end{cor}

\medskip

Next we define a family of realizations of $\cL$ in $L^2(\Omega)$ with general Robin-type
boundary conditions of the form
\begin{equation}\label{AB4}
  A_{[B]}f \defeq \cL f, \quad
  \dom A_{[B]} \defeq \bigl\{ f\in H^{3/2}_\cL(\Omega)\colon
  Bf|_{\partial\Omega} = \partial_\cL f|_{\partial\Omega}\bigr\},
\end{equation}
where $B$ is a bounded self-adjoint operator in $L^2(\partial\Omega)$.
In terms of the quasi boundary triple in Proposition~\ref{prop:qbt}
the operator  $A_{[B]}$ coincides with the one in \eqref{AB0}, which
is also equal to the restriction
\[
  T\upharpoonright\ker(B\Gamma_1-\Gamma_0).
\]
The following corollary is a consequence of Theorem~\ref{thm:sa}
since $\ran\Gamma_0=L^2(\partial\Omega)$, $\AD$ is self-adjoint and $M(\lambda)$ is
compact for $\lambda\in\rho(\AN)$ by Proposition~\ref{prop:Sp_g_M}\,(ii).

\begin{cor}
\label{thm:sa2}
Let $\{L^2(\partial\Omega),\Gamma_0,\Gamma_1\}$ be the quasi boundary triple from
Proposition~\ref{prop:qbt} with $\gamma$-field $\gamma$ and Weyl function $M$, and
let $B$ be a bounded self-adjoint operator in $L^2(\partial\Omega)$.
Then the corresponding operator $A_{[B]}$ in \eqref{AB4} is self-adjoint in $L^2(\Omega)$ and
\begin{align}
  (A_{[B]}-\lambda)^{-1} - (\AN-\lambda)^{-1}
  &= \gamma(\lambda)\bigl(I-BM(\lambda)\bigr)^{-1}B\gamma(\ov\lambda)^* \label{krein1} \\[0.5ex]
  &= \gamma(\lambda)B\bigl(I-M(\lambda)B\bigr)^{-1}\gamma(\ov\lambda)^* \label{krein2}
\end{align}
holds for $\lambda\in\rho(A_{[B]})\cap\rho(\AN)$ with
\begin{equation}\label{robin_to_neumann}
  \bigl(I-BM(\lambda)\bigr)^{-1}, \bigl(I- M(\lambda)B\bigr)^{-1}
  \in\cB\bigl(L^2(\partial\Omega)\bigr).
\end{equation}
\end{cor}

\medskip

\noindent
Note that the operators in \eqref{robin_to_neumann} can be viewed as Robin-to-Neumann maps.

\subsection{Operator ideal properties and traces of resolvent power differences}
\label{ssec:ideal}
In this subsection we prove the main results of this note: estimates for the singular values
of resolvent power differences of two self-adjoint realizations of the differential expression $\cL$
subject to Dirichlet, Neumann and non-local Robin boundary conditions.

The first theorem on the difference of the resolvent powers of the Dirichlet and Neumann operator
is partially known from \cite{B62} and \cite{G84,M10}, where the proof is based on
variational principles, pseudo-differential methods or a reduction to higher order operators.
Here we give an elementary, direct proof using our approach.
In the case of first powers of the resolvents, the trace formula in item (ii) is
contained in \cite{AB09, BLL11}.  An equivalent formula can also be found in \cite{Ca02},
where it is used for the analysis of the Laplace--Beltrami operator on coupled manifolds.

\begin{thm}\label{thm1}
Let $\AD$ and $\AN$ be the self-adjoint Dirichlet and Neumann realization of $\cL$ in \eqref{ADAN} and let
$M$ be the Weyl function from Proposition~\ref{prop:qbt}.
Then the following statements hold.
\begin{itemize}\setlength{\itemsep}{1.2ex}
\item [\rm (i)]
For all $m\in\dN$ and $\lambda\in\rho(\AN)\cap\rho(\AD)$,
\begin{equation}\label{Sp0}
  (\AN-\lambda)^{-m} - (\AD-\lambda)^{-m} \in \sS_{\frac{n-1}{2m},\infty}\bigl(L^2(\Omega)\bigr).
\end{equation}
\item [\rm (ii)]
If $m > \tfrac{n-1}{2}$ then
the resolvent power difference in \eqref{Sp0} is a trace class operator and, for all $\lambda\in\rho(\AN)\cap\rho(\AD)$,
\begin{equation*}
  \tr\Bigl((\AN-\lambda)^{-m} - (\AD-\lambda)^{-m}\Bigr)
  = \frac{1}{(m-1)!}\tr\Biggl(\frac{d^{m-1}}{d\lambda^{m-1}}
  \Bigl(M(\lambda)^{-1}M'(\lambda)\Bigr)\Biggr).
\end{equation*}
\end{itemize}
\end{thm}

\begin{proof}
(i)
The proof of the first item is carried out in two steps. \\
\textit{Step 1.}
Let us introduce the operator function
\[
  S(\lambda) \defeq M(\lambda)^{-1}\gamma(\ov\lambda)^*, \qquad \lambda\in\rho(\AN)\cap\rho(\AD).
\]
Note that the product is well defined since
$\ran(\gamma(\ov\lambda)^*) \subset H^1(\partial\Omega) = \dom(M(\lambda)^{-1})$.
Since $\AD$ is self-adjoint, it follows from Proposition~\ref{gammaprop}\,(iii)
that $S(\lambda)$ is a bounded operator from $L^2(\Omega)$ to $L^2(\partial\Omega)$
for $\lambda\in\rho(\AN)\cap\rho(\AD)$.
We prove the following smoothing property for the derivatives of $S$:
\begin{equation}\label{Sk_smoothing}
  u\in H^s_{\rm loc}(\ov{\Omega})\cap L^2(\Omega) \quad\Rightarrow\quad
  S^{(k)}(\lambda)u\in H^{s+2k+1/2}(\partial\Omega), \qquad
  s\ge 0,\, k\in\NN_0,
\end{equation}
by induction.  Since $\gamma(\ov\lambda)^*$ maps $H^s_{\rm loc}(\ov{\Omega})\cap L^2(\Omega)$ into $H^{s+3/2}(\partial\Omega)$
for $s\ge0$ by Lemma~\ref{lem:gammaWeyl}\,(ii) and $M(\lambda)^{-1}$ maps $H^{s+3/2}(\partial\Omega)$
into $H^{s+1/2}(\partial\Omega)$ by Lemma~\ref{lem:gammaWeyl}\,(iv), relation
\eqref{Sk_smoothing} is true for $k=0$.
Now let $l\in\NN_0$ and assume that \eqref{Sk_smoothing} is true for every $k=0,1,\dots,l$.
By \eqref{rule1}, \eqref{rule3} and Lemma~\ref{le:der_g_M}\,(i), (iii) we have
\begin{align*}
  S'(\lambda)u
  &= \frac{d}{d\lambda}\big(M(\lambda)^{-1}\big)\gamma(\ov\lambda)^*u
  + M(\lambda)^{-1}\frac{d}{d\lambda}\gamma(\ov\lambda)^*u  \\[0.5ex]
  &= -M(\lambda)^{-1}M'(\lambda)M(\lambda)^{-1}\gamma(\ov\lambda)^*u + M(\lambda)^{-1}\gamma(\ov\lambda)^*(\AN-\lambda)^{-1}u  \\[0.5ex]
  &= -M(\lambda)^{-1}\gamma(\ov\lambda)^*\gamma(\lambda)M(\lambda)^{-1}\gamma(\ov\lambda)^*u
  + S(\lambda)(\AN-\lambda)^{-1}u \\[0.5ex]
  &= S(\lambda)(\AN-\lambda)^{-1}u - S(\lambda)\gamma(\lambda)S(\lambda)u
\end{align*}
for all $u\in L^2(\Omega)$.
Hence, with the help of \eqref{rule1}, \eqref{rule2} and Lemma~\ref{le:der_g_M}\,(ii),
we obtain
\begin{align}
  S^{(l+1)}(\lambda)
  =& \frac{d^l}{d\lambda^l}\Bigl(S(\lambda)(\AN-\lambda)^{-1}
  - S(\lambda)\gamma(\lambda)S(\lambda)\Bigr) \notag\\[0.5ex]
  =& \sum_{\substack{p+q=l \\[0.2ex] p,q\ge0}} \binom{l}{p}
  S^{(p)}(\lambda)\frac{d^q}{d\lambda^q}(\AN-\lambda)^{-1} \notag\\[0.5ex]
  &- \sum_{\substack{p+q+r=l \\[0.2ex] p,q,r\ge0}} \frac{l!}{p!\,q!\,r!}
  S^{(p)}(\lambda)\gamma^{(q)}(\lambda)S^{(r)}(\lambda)
  \notag\\[0.5ex]
  \label{expans_Slpl1}
  =& \sum_{\substack{p+q=l \\[0.2ex] p,q\ge0}} \frac{l!}{p!}S^{(p)}(\lambda)(\AN-\lambda)^{-(q+1)}
  \\&- \sum_{\substack{p+q+r=l \\[0.2ex] p,q,r\ge0}} \frac{l!}{p!\,r!}
  S^{(p)}(\lambda)(\AN-\lambda)^{-q}\gamma(\lambda)S^{(r)}(\lambda).\notag
\end{align}
By the induction hypothesis, the smoothing property \eqref{elliptic} and
Lemma~\ref{lem:gammaWeyl}\,(i), we have, for $s\ge0$ and $p,q\geq 0$, $p+q=l$,
\begin{align*}
  &u\in H^s_{\rm loc}(\ov{\Omega})\cap L^2(\Omega) \\[0.5ex]
  &\Longrightarrow\quad (\AN-\lambda)^{-(q+1)}u \in H^{s+2q+2}_{\rm loc}(\ov{\Omega})\cap L^2(\Omega) \\[0.5ex]
  &\Longrightarrow\quad S^{(p)}(\lambda)(\AN-\lambda)^{-(q+1)}u
  \in H^{s+2q+2+2p+1/2}(\partial\Omega) = H^{s+2(l+1)+1/2}(\partial\Omega)
\end{align*}
and for $s\ge0$ and $p,q,r\geq 0$, $p+q+r=l$,
\begin{align*}
  &u\in H^s_{\rm loc}(\ov{\Omega})\cap L^2(\Omega)\\[0.5ex]
  &\Longrightarrow\quad S^{(r)}(\lambda)u \in H^{s+2r+1/2}(\partial\Omega) \\[0.5ex]
  &\Longrightarrow\quad \gamma(\lambda)S^{(r)}(\lambda)u
    \in H^{s+2r+1/2+3/2}_{\rm loc}(\ov{\Omega})\cap L^2(\Omega)\\[0.5ex]
  &\Longrightarrow\quad (\AN-\lambda)^{-q}\gamma(\lambda)S^{(r)}(\lambda)u
    \in H^{s+2r+2+2q}_{\rm loc}(\ov{\Omega})\cap L^2(\Omega) \\[0.5ex]
  &\Longrightarrow\quad S^{(p)}(\lambda)(\AN-\lambda)^{-q}\gamma(\lambda)S^{(r)}(\lambda)u
    \in H^{s+2r+2+2q+2p+1/2}(\partial\Omega) \\[0.5ex]
  &\hspace*{42ex} = H^{s+2(l+1)+1/2}(\partial\Omega),
\end{align*}
which, together with \eqref{expans_Slpl1}, shows \eqref{Sk_smoothing} for $k=l+1$
and hence, by induction, for all $k\in\NN_0$.
Therefore, an application of Lemma~\ref{le.s_emb} yields that
\begin{equation}\label{Sk_in_Sp}
  S^{(k)}(\lambda) \in \frS_{\frac{n-1}{2k+1/2},\infty}\bigl(L^2(\Omega),L^2(\partial\Omega)\bigr),
  \qquad k\in\NN_0,\,\lambda\in\rho(\AN)\cap\rho(\AD).
\end{equation}
\textit{Step 2.}
Using Krein's formula from Corollary~\ref{cor:01} and \eqref{rule1} we can write,
for $m\in\NN$ and $\lambda\in\rho(\AN)\cap\rho(\AD)$,
\begin{align}
  (\AN-\lambda)^{-m}-(\AD-\lambda)^{-m}
  &= \frac{1}{(m-1)!}\cdot\frac{d^{m-1}}{d\lambda^{m-1}}\Bigl((\AN-\lambda)^{-1}-(\AD-\lambda)^{-1}\Bigr)
    \notag\\[0.5ex]
  &= \frac{1}{(m-1)!}\cdot\frac{d^{m-1}}{d\lambda^{m-1}}\bigl(\gamma(\lambda)S(\lambda)\bigr)
    \notag\\[0.5ex]
  &= \frac{1}{(m-1)!}\sum_{\substack{p+q=m-1 \\[0.2ex] p,q\ge0}}
  \binom{m-1}{p} \gamma^{(p)}(\lambda) S^{(q)}(\lambda).
    \label{sum578}
\end{align}
Since, by Proposition~\ref{prop:Sp_g_M}\,(i), \eqref{Sk_in_Sp} and \eqref{prod_Sp},
\begin{equation}\label{terms_in_Sp1}
  \gamma^{(p)}(\lambda)S^{(q)}(\lambda)
  \in \frS_{\frac{n-1}{2p+3/2},\infty}\cdot\frS_{\frac{n-1}{2q+1/2},\infty}
  = \frS_{\frac{n-1}{2(p+q)+2},\infty} = \frS_{\frac{n-1}{2m},\infty}
\end{equation}
for $p,q$ with $p+q=m-1$, we obtain \eqref{Sp0}.

(ii)
If $m>\tfrac{n-1}{2}$ then $\frac{n-1}{2m}<1$ and, by \eqref{Sp_con_Spinf}
and \eqref{terms_in_Sp1}, each term in the sum in \eqref{sum578} is a trace class
operator and, by a similar argument, also $S^{(q)}(\lambda)\gamma^{(p)}(\lambda)$.
Hence the operator in \eqref{Sp0} is a trace class operator, and
we can apply the trace to \eqref{sum578} and use \eqref{trace1}, \eqref{trace2} and
Lemma~\ref{le:der_g_M}\,(iii) to obtain
\begin{align*}
  & (m-1)!\tr\Bigl((\AN-\lambda)^{-m}-(\AD-\lambda)^{-m}\Bigr)
  = \tr\Biggl(\,\sum_{\substack{p+q=m-1 \\[0.2ex] p,q\ge0}}\binom{m-1}{p}
  \gamma^{(p)}(\lambda) S^{(q)}(\lambda)\Biggr) \\[0.5ex]
  &= \sum_{\substack{p+q=m-1 \\[0.2ex] p,q\ge0}}\binom{m-1}{p}
  \tr\Bigl(\gamma^{(p)}(\lambda) S^{(q)}(\lambda)\Bigr)
  = \sum_{\substack{p+q=m-1 \\[0.2ex] p,q\ge0}}\binom{m-1}{p}
  \tr\Bigl(S^{(q)}(\lambda)\gamma^{(p)}(\lambda)\Bigr) \displaybreak[0]\\[0.5ex]
  &= \tr\Biggl(\,\sum_{\substack{p+q=m-1 \\[0.2ex] p,q\ge0}}\binom{m-1}{p}
  S^{(q)}(\lambda)\gamma^{(p)}(\lambda)\Biggr)
  = \tr\biggl(\frac{d^{m-1}}{d\lambda^{m-1}}\Bigl(S(\lambda)\gamma(\lambda)\Bigr)\biggr) \\[0.5ex]
  &= \tr\biggl(\frac{d^{m-1}}{d\lambda^{m-1}}
  \Bigl(M(\lambda)^{-1}\gamma(\ov\lambda)^*\gamma(\lambda)\Bigr)\biggr)
  = \tr\biggl(\frac{d^{m-1}}{d\lambda^{m-1}}
  \Bigl(M(\lambda)^{-1}M'(\lambda)\Bigr)\biggr),
\end{align*}
which finishes the proof.
\end{proof}


In the following theorem, which contains the main result of this note, we prove
weak Schatten--von Neumann estimates for resolvent power differences of two
self-adjoint realizations $A_{[B_1]}$ and $A_{[B_2]}$ of $\cL$ with Robin and more general
non-local boundary conditions.  In this situation the estimates are better than for the
pair of Dirichlet and Neumann realizations in Theorem~\ref{thm1}.  For the first powers of
the resolvents this was already observed in \cite{BLLLP10,BLL11} and \cite{G11-2}.
In the special important case when the resolvent power difference is a trace class operator
we express its trace as the trace of a certain operator acting on the boundary $\partial\Omega$,
which is given in terms of the Weyl function and the operators $B_1$ and $B_2$ in the
boundary conditions; cf.\ \cite[Corollary~4.12]{BLL11} for the case of first powers
and \cite{BMN08,GZ12} for one-dimensional Schr\"odinger operators and other
finite-dimensional situations.  We also mention that the special case of
classical Robin boundary conditions, where $B_1$ and $B_2$ are multiplication operators
with real-valued $L^\infty$-functions is contained in the theorem.

\begin{thm}\label{thm2}
Let $\{L^2(\partial\Omega),\Gamma_0,\Gamma_1\}$ be the quasi boundary triple from
Proposition~\ref{prop:qbt} with Weyl function $M$ and let $\AN$ be the self-adjoint Neumann
operator in \eqref{ADAN}.  Moreover, let $B_1$ and $B_2$ be bounded self-adjoint
operators in $L^2(\partial\Omega)$, define $A_{[B_1]}$ and $A_{[B_2]}$ as in \eqref{AB4} and set
\[
  t \defeq \begin{cases}
    \dfrac{n-1}{s} & \text{if}\; B_1-B_2 \in \frS_{s,\infty}(L^2(\partial\Omega)) \text{ for some }s>0, \\[1ex]
    0 & \text{otherwise}.
  \end{cases}
\]
Then the following statements hold.
\begin{itemize}\setlength{\itemsep}{1.2ex}
\item[\rm (i)] For all $m\in\dN$ and $\lambda\in\rho(A_{[B_1]})\cap\rho(A_{[B_2]})$,
\begin{equation}\label{Sp}
  (A_{[B_1]}-\lambda)^{-m} - (A_{[B_2]}-\lambda)^{-m} \in \sS_{\frac{n-1}{2m+t+1},\infty}\bigl(L^2(\Omega)\bigr).
\end{equation}
\item[\rm (ii)] If $m>\tfrac{n-t}{2}-1$
then the resolvent power difference in \eqref{Sp} is a trace class operator and,
for all $\lambda\in\rho(A_{[B_1]})\cap\rho(A_{[B_2]})\cap\rho(A_{\rm N})$,
\begin{equation}\label{tr}
  \tr\Bigl((A_{[B_1]}-\lambda)^{-m} - (A_{[B_2]}-\lambda)^{-m}\Bigr)
  = \frac{1}{(m-1)!}\tr\biggl(\frac{d^{m-1}}{d\lambda^{m-1}}\Bigl(U(\lambda)M'(\lambda)\Bigr)\biggr)
\end{equation}
where  $U(\lambda) \defeq \bigl(I-B_1 M(\lambda)\bigr)^{-1}(B_1-B_2)\bigl(I- M(\lambda)B_2\bigr)^{-1}$.
\end{itemize}
\end{thm}

\begin{proof}
(i)
In order to shorten notation and to avoid the distinction of several cases,
we set
\[
  \frA_r \defeq \begin{cases}
    \frS_{\frac{n-1}{r},\infty}\bigl(L^2(\partial\Omega)\bigr) & \text{if}\; r>0, \\[1ex]
    \cB\bigl(L^2(\partial\Omega)\bigr) & \text{if}\; r=0.
  \end{cases}
\]
It follows from \eqref{prod_Sp} and the fact that $\frS_{p,\infty}(L^2(\partial\Omega))$, $p>0$,
is an ideal in $\cB(L^2(\partial\Omega))$ that
\begin{equation}\label{prodAr}
  \frA_{r_1}\cdot\frA_{r_2} = \frA_{r_1+r_2}, \qquad r_1,r_2\ge0.
\end{equation}
Moreover, the assumption on the difference of $B_1$ and $B_2$ yields
\begin{equation}\label{B1B2inAt}
  B_1-B_2 \in \frA_t.
\end{equation}
The proof of item (i) is divided into three steps. \\[0.5ex]
\textit{Step 1.}
Let $B$ be a bounded self-adjoint operator in $L^2(\partial\Omega)$ and set
\[
  T(\lambda) \defeq \bigl(I-BM(\lambda)\bigr)^{-1},
  \qquad \lambda\in\rho(A_{[B]})\cap\rho(\AN),
\]
where $T(\lambda)\in\cB(L^2(\partial\Omega))$ by Corollary~\ref{thm:sa2}.
We show that
\begin{equation}\label{TkinAk}
  T^{(k)}(\lambda) \in \frA_{2k+1}, \qquad k\in\NN,
\end{equation}
by induction.  Relation \eqref{rule3} implies that
\begin{equation}\label{Tder}
  T'(\lambda) = T(\lambda)BM'(\lambda)T(\lambda),
\end{equation}
which is in $\frA_3$ by Proposition~\ref{prop:Sp_g_M}\,(ii).
Let $l\in\NN$ and assume that \eqref{TkinAk} is true for every $k=1,\dots,l$, which
implies in particular that
\begin{equation}\label{TkinA2k}
  T^{(k)}(\lambda) \in \frA_{2k}, \qquad k=0,\dots,l.
\end{equation}
Then
\[
  T^{(l+1)}(\lambda)
  =\frac{d^l}{d\lambda^l}\Bigl(T(\lambda)BM'(\lambda)T(\lambda)\Bigr)
  =\sum_{\substack{p+q+r=l \\[0.2ex] p,q,r\ge0}}\frac{l!}{p!\,q!\,r!}
  T^{(p)}(\lambda)BM^{(q+1)}(\lambda)T^{(r)}(\lambda)
\]
by \eqref{Tder} and \eqref{rule2}.  Relation \eqref{TkinA2k}, the boundedness of $B$,
Proposition~\ref{prop:Sp_g_M}\,(ii) and \eqref{prodAr} imply that
\[
  T^{(p)}(\lambda)BM^{(q+1)}(\lambda)T^{(r)}(\lambda)
  \in \frA_{2p}\cdot\frA_{2(q+1)+1}\cdot\frA_{2r}
  = \frA_{2(l+1)+1}
\]
since $p+q+r=l$.  This shows \eqref{TkinAk} for $k=l+1$ and hence, by induction,
for all $k\in\NN$.
Since $T(\lambda)\in\cB(L^2(\partial\Omega))$, we have
\begin{equation}\label{TkinAk_weak}
  T^{(k)}(\lambda) \in \frA_{2k}, \qquad k\in\NN_0,\;\lambda\in\rho(\AN),
\end{equation}
and by similar considerations also
\begin{equation}\label{TkinAk_rev}
  \frac{d^k}{d\lambda^k}\bigl(I-M(\lambda)B\bigr)^{-1} \in \frA_{2k},
  \qquad k\in\NN_0,\;\lambda\in\rho(\AN).
\end{equation}
\textit{Step 2.}
With $B_1$, $B_2$ as in the statement of the theorem set
\[
  T_1(\lambda) \defeq \bigl(I-B_1M(\lambda)\bigr)^{-1}\quad\text{and}\quad
  T_2(\lambda) \defeq \bigl(I-M(\lambda)B_2\bigr)^{-1}
\]
for $\lambda\in\rho(A_{[B_1]})\cap\rho(A_{[B_2]})\cap\rho(\AN)$.
We can write $U(\lambda)=T_1(\lambda)(B_1-B_2)T_2(\lambda)$ and hence
\[
  U^{(k)}(\lambda) = \frac{d^k}{d\lambda^k}\Bigl(T_1(\lambda)(B_1-B_2)T_2(\lambda)\Bigr)
  = \sum_{\substack{p+q=k \\[0.2ex] p,q\ge0}}\binom{k}{p}
  T_1^{(p)}(\lambda)(B_1-B_2)T_2^{(q)}(\lambda).
\]
By \eqref{TkinAk_weak}, \eqref{TkinAk_rev} and \eqref{B1B2inAt},
each term in the sum satisfies
\[
  T_1^{(p)}(\lambda)(B_1-B_2)T_2^{(q)}(\lambda)
  \in \frA_{2p}\cdot\frA_t\cdot\frA_{2q} = \frA_{2k+t},
\]
and hence
\begin{equation}\label{UinAk}
  U^{(k)}(\lambda)\in\frA_{2k+t}, \qquad k\in\NN_0,\;\lambda\in\rho(\AN).
\end{equation}

\textit{Step 3.}
By applying \eqref{krein1} to $A_{[B_1]}$ and \eqref{krein2} to $A_{[B_2]}$
and taking the difference we obtain that, for
$\lambda\in\rho(A_{[B_1]})\cap\rho(A_{[B_2]})\cap\rho(\AN)$,
\begin{align*}
  & (A_{[B_1]}-\lambda)^{-1} - (A_{[B_2]}-\lambda)^{-1} \\[0.5ex]
  &= \gamma(\lambda)\Bigl[\bigl(I-B_1M(\lambda)\bigr)^{-1}B_1
  -B_2\bigl(I-M(\lambda)B_2\bigr)^{-1}\Bigr]\gamma(\ov\lambda)^* \displaybreak[0]\\[0.5ex]
  &= \gamma(\lambda)\Bigl[\bigl(I-B_1M(\lambda)\bigr)^{-1}B_1
  \bigl(I-M(\lambda)B_2\bigr)\bigl(I-M(\lambda)B_2\bigr)^{-1} \\[0.5ex]
  &\quad - \bigl(I-B_1M(\lambda)\bigr)^{-1}\bigl(I-B_1M(\lambda)\bigr)B_2
  \bigl(I-M(\lambda)B_2\bigr)^{-1}\Bigr]\gamma(\ov\lambda)^* \\[0.5ex]
  &= \gamma(\lambda)\Bigl[\bigl(I-B_1M(\lambda)\bigr)^{-1}(B_1-B_2)
  \bigl(I-M(\lambda)B_2\bigr)^{-1}\Bigr]\gamma(\ov\lambda)^*
  = \gamma(\lambda)U(\lambda)\gamma(\ov\lambda)^*.
\end{align*}
Taking derivatives we get, for $m\in\NN$,
\begin{align}
  & (A_{[B_1]}-\lambda)^{-m} - (A_{[B_2]}-\lambda)^{-m} \notag\\[0.5ex]
  &= \frac{1}{(m-1)!}\cdot\frac{d^{m-1}}{d\lambda^{m-1}}
  \Bigl((A_{[B_1]}-\lambda)^{-1} - (A_{[B_2]}-\lambda)^{-1}\Bigr) \notag\displaybreak[0]\\[0.5ex]
  &= \frac{1}{(m-1)!}\cdot\frac{d^{m-1}}{d\lambda^{m-1}}
  \Bigl(\gamma(\lambda)U(\lambda)\gamma(\ov\lambda)^*\Bigr) \notag\\[0.5ex]
  &= \frac{1}{(m-1)!}\sum_{\substack{p+q+r=m-1 \\[0.2ex] p,q,r\ge0}}\frac{(m-1)!}{p!\,q!\,r!}
  \gamma^{(p)}(\lambda)U^{(q)}(\lambda)\frac{d^r}{d\lambda^r}\gamma(\ov\lambda)^*.
  \label{594}
\end{align}
By Proposition~\ref{prop:Sp_g_M}\,(i) and \eqref{UinAk}, each term in the sum satisfies
\begin{equation}\label{terms_in_Sp2}
  \gamma^{(p)}(\lambda)U^{(q)}(\lambda)\frac{d^r}{d\lambda^r}\gamma(\ov\lambda)^*
  \in \frS_{\frac{n-1}{2p+3/2},\infty}\cdot\frS_{\frac{n-1}{2q+t},\infty}
  \cdot\frS_{\frac{n-1}{2r+3/2},\infty}
  = \frS_{\frac{n-1}{2m+t+1},\infty},
\end{equation}
which proves \eqref{Sp}.

(ii)
If $m>\frac{n-t}{2}-1$ then $\frac{n-1}{2m+t+1}<1$ and, by \eqref{Sp_con_Spinf}
and \eqref{terms_in_Sp2}, all terms in the sum in \eqref{594}
are trace class operators, and the same is true if we change the order in the product in \eqref{terms_in_Sp2}.
Hence we can apply the trace to the expression in \eqref{594} and
use \eqref{trace1}, \eqref{trace2} and Lemma~\ref{le:der_g_M}\,(iii) to obtain
\begin{align*}
  & (m-1)!\tr\Bigl((A_{[B_1]}-\lambda)^{-m} - (A_{[B_2]}-\lambda)^{-m}\Bigr) \\[0.5ex]
  &= \tr\Biggl(\,\sum_{\substack{p+q+r=m-1 \\[0.2ex] p,q,r\ge0}}\frac{(m-1)!}{p!\,q!\,r!}
  \gamma^{(p)}(\lambda)U^{(q)}(\lambda)\frac{d^r}{d\lambda^r}\gamma(\ov\lambda)^*\Biggr) \\[0.5ex]
  &= \sum_{\substack{p+q+r=m-1 \\[0.2ex] p,q,r\ge0}} \frac{(m-1)!}{p!\,q!\,r!}
  \tr\Bigl(\gamma^{(p)}(\lambda)U^{(q)}(\lambda)\frac{d^r}{d\lambda^r}\gamma(\ov\lambda)^*\Bigr) \\[0.5ex]
  &= \sum_{\substack{p+q+r=m-1 \\[0.2ex] p,q,r\ge0}} \frac{(m-1)!}{p!\,q!\,r!}
  \tr\biggl(U^{(q)}(\lambda)\Bigl(\frac{d^r}{d\lambda^r}\gamma(\ov\lambda)^*\Bigr)
  \gamma^{(p)}(\lambda)\biggr) \displaybreak[0]\\[0.5ex]
  &= \tr\Biggl(\,\sum_{\substack{p+q+r=m-1 \\[0.2ex] p,q,r\ge0}}\frac{(m-1)!}{p!\,q!\,r!}
  U^{(q)}(\lambda)\Bigl(\frac{d^r}{d\lambda^r}\gamma(\ov\lambda)^*\Bigr)
  \gamma^{(p)}(\lambda)\Biggr) \\[0.5ex]
  &= \tr\biggl(\frac{d^{m-1}}{d\lambda^{m-1}}\Bigl(U(\lambda)\gamma(\ov\lambda)^*
  \gamma(\lambda)\Bigr)\biggr)
  = \tr\biggl(\frac{d^{m-1}}{d\lambda^{m-1}}\Bigl(U(\lambda)M'(\lambda)\Bigr)\biggr),
\end{align*}
which shows \eqref{tr}.
\end{proof}

\begin{rem}
The statements of Theorem~\ref{thm2} remain true if $A$ is an arbitrary closed
symmetric operator in a Hilbert space $\cH$ and $\{\cG,\Gamma_0,\Gamma_1\}$ a
quasi boundary triple for $A^*$ such that $\ran\Gamma_0=\cG$ and the statements
of Proposition~\ref{prop:Sp_g_M} are true with $L^2(\Omega)$ and $L^2(\partial\Omega)$
replaced by $\cH$ and $\cG$, respectively.
\end{rem}

As a special case of the last theorem let us consider the situation when $B_1 = B$
and $B_2 = 0$, where $B$ is a bounded self-adjoint operator in $L^2(\partial\Omega)$.
This immediately leads to the following corollary.

\begin{cor}\label{cor.robin_neumann}
Let $\{L^2(\partial\Omega),\Gamma_0,\Gamma_1\}$ be the quasi boundary triple from
Proposition~\ref{prop:qbt} with Weyl function $M$ and let $\AN$ be the self-adjoint
Neumann operator in \eqref{ADAN}.  Moreover, let $B$ be a bounded self-adjoint
operator in $L^2(\partial\Omega)$, define $A_{[B]}$ as in \eqref{AB4} and set
\[
  t \defeq \begin{cases}
    \dfrac{n-1}{s} & \text{if}\; B \in \frS_{s,\infty}(L^2(\partial\Omega)) \text{ for some }s>0, \\[1ex]
    0 & \text{otherwise}.
  \end{cases}
\]
Then the following statements hold.
\begin{itemize}\setlength{\itemsep}{1.2ex}
\item [\rm (i)] For all $m\in\dN$ and $\lambda\in\rho(A_{[B]})\cap\rho(\AN)$,
\begin{equation*}
\label{Sp4}
  (A_{[B]}-\lambda)^{-m} - (\AN-\lambda)^{-m}
  \in\sS_{\frac{n-1}{2m+t+1},\infty}\bigl(L^2(\Omega)\bigr),
\end{equation*}
\item [\rm (ii)] If $m>\tfrac{n-t}{2}-1$ then the resolvent power
difference in \eqref{Sp3} is a trace class operator and, for
all $\lambda\in\rho(A_{[B]})\cap\rho(\AN)$,
\begin{equation*}
\label{tr4}
\begin{aligned}
  &\tr\Bigl((A_{[B]}-\lambda)^{-m} - (\AN-\lambda)^{-m}\Bigr) \\[0.5ex]
  &\qquad= \frac{1}{(m-1)!}\tr\Biggl(\frac{d^{m-1}}{d\lambda^{m-1}}
  \Bigl(\big(I- BM(\lambda)\big)^{-1}BM'(\lambda)\Bigr)\Biggr).
\end{aligned}
\end{equation*}
\end{itemize}
\end{cor}

\medskip

The following theorem, where we compare operators with non-local and Dirichlet boundary conditions,
is a consequence of Theorems~\ref{thm1} and \ref{thm2}.

\begin{thm}\label{thm3}
Let $\{L^2(\partial\Omega),\Gamma_0,\Gamma_1\}$ be the quasi boundary triple from
Proposition~\ref{prop:qbt} with Weyl function $M$ and let $\AD$ be the self-adjoint
Dirichlet operator in \eqref{ADAN}.  Moreover, let $B$ be a bounded self-adjoint
operator in $L^2(\partial\Omega)$ and define $A_{[B]}$ as in \eqref{AB4}.
Then the following statements hold.
\begin{itemize}\setlength{\itemsep}{1.2ex}
\item [\rm (i)] For all $m\in\dN$ and $\lambda\in\rho(A_{[B]})\cap\rho(\AD)$,
\begin{equation}
\label{Sp3}
  (A_{[B]}-\lambda)^{-m} - (\AD-\lambda)^{-m}
  \in\sS_{\frac{n-1}{2m},\infty}\big(L^2(\Omega)\big).
\end{equation}
\item [\rm (ii)] If $m>\tfrac{n-1}{2}$ then the resolvent power
difference in \eqref{Sp3} is a trace class operator and, for
all $\lambda\in\rho(A_{[B]})\cap\rho(\AD)\cap\rho(\AN)$,
\begin{equation}
\label{tr3}
  \tr\Bigl((A_{[B]}-\lambda)^{-m} - (\AD-\lambda)^{-m}\Big)
  = \frac{1}{(m-1)!}\tr\biggl(\frac{d^{m-1}}{d\lambda^{m-1}}
  \Big(V(\lambda)M'(\lambda)\Big)\biggr)
\end{equation}
where $V(\lambda) \defeq \bigl(I-M(\lambda)B\bigr)^{-1}M(\lambda)^{-1}$.
\end{itemize}
\end{thm}

\begin{proof}
(i)
Let us fix $\lambda\in\rho(A_{[B]})\cap\rho(\AD)\cap\rho(\AN)$.
From Theorems~\ref{thm1}\,(i) and \ref{thm2}\,(i) it follows that
\begin{align*}
  X_1(\lambda) \defequ (\AN-\lambda)^{-m} - (\AD-\lambda)^{-m}
  \in\sS_{\frac{n-1}{2m},\infty}, \\[1ex]
  X_2(\lambda) \defequ (A_{[B]}-\lambda)^{-m} - (\AN-\lambda)^{-m}
  \in\sS_{\frac{n-1}{2m+1},\infty}\subset\sS_{\frac{n-1}{2m},\infty},
\end{align*}
and thus
\[
  (A_{[B]}-\lambda)^{-m} - (\AD-\lambda)^{-m} = X_1(\lambda) + X_2(\lambda)
  \in \sS_{\frac{n-1}{2m},\infty}.
\]
By analyticity we can extend this to all points $\lambda$ in $\rho(A_{[B]})\cap\rho(\AD)$.

(ii) If $m > \tfrac{n-1}{2}$, then $\frac{n-1}{2m}<1$ and hence, by item (i)
and \eqref{Sp_con_Spinf}, the operator in \eqref{Sp3} is a trace class operator.
Using Theorem~\ref{thm1}\,(ii) and Corollary~\ref{cor.robin_neumann}\,(ii) we obtain
\begin{align*}
  &\tr\Bigl((A_{[B]}-\lambda)^{-m} - (\AD-\lambda)^{-m}\Bigr)
  = \tr\bigl(X_1(\lambda)  +  X_2(\lambda)\bigr) \\
  &= \frac{1}{(m-1)!}\tr\Biggl(\frac{d^{m-1}}{d\lambda^{m-1}}
  \biggl[\Bigl(M(\lambda)^{-1} + \bigl(I-BM(\lambda)\bigr)^{-1}B\Bigr)M'(\lambda)\biggr]\Biggr).
\end{align*}
Since
\begin{align*}
  & M(\lambda)^{-1} + \bigl(I-BM(\lambda)\bigr)^{-1}B \\[0.5ex]
  &= \bigl(I-BM(\lambda)\bigr)^{-1}\Bigl[\bigl(I-BM(\lambda)\bigr)+BM(\lambda)\Bigr]M(\lambda)^{-1}
  = V(\lambda),
\end{align*}
this implies \eqref{tr3}.
\end{proof}

Note that, for $B$ being a multiplication operator by a bounded function $\beta$, the statement in (i)
of the previous theorem is exactly the estimate \eqref{est2}.


\end{document}